\documentclass[12pt]{amsart}
\usepackage{latexsym,amssymb,amsfonts,amsmath, amscd, euscript, MnSymbol}
\usepackage{pdfsync}
\usepackage{enumerate}
\usepackage[pdftex]{graphicx}
\usepackage{wrapfig}
\usepackage{epsfig}
\usepackage{epstopdf}
\usepackage{enumerate}
\usepackage{verbatim}
\usepackage{float}

\newtheorem{theorem}{Theorem}

\newtheorem{corollary}[theorem]{Corollary}

\newtheorem{lemma}{Lemma}
\newtheorem{definitions}[lemma]{Definitions}

\newtheorem{fact}{Fact}

\newtheorem{problem}[theorem]{Problem}
\newtheorem{remark}[theorem]{Remark}

\DeclareMathOperator{\Min}{Min}

\title[Free monoids and  generalized metric spaces]{Free monoids and generalized metric spaces}

\author[M.Kabil]{Mustapha Kabil}
\address{Laboratoire Math\'ematiques et Applications, D\'epartement de Math{\'e}matiques, Facult\'e des Sciences et Techniques, Universit\'e Hassan II -Casablanca, BP 146 Mohammedia, Morocco. }
\email{kabilfstm@gmail.com}

\author [M. Pouzet]{Maurice Pouzet**}
\address{Univ. Lyon, Universit\'e Claude-Bernard Lyon1, CNRS UMR 5208, Institut Camille Jordan,  43 bd. 11 Novembre 1918, 69622 Villeurbanne Cedex, France and Mathematics \& Statistics Department, University of Calgary, Calgary, Alberta, Canada T2N 1N4}
\email{pouzet@univ-lyon1.fr, mpouzet@ucalgary.ca }

\thanks{Supports provided  by PHC France-Maghreb 14MAG14, ICJ of University Claude-Bernard and  University Hassan II are gratefully acknowledged}

 \author [I.G.Rosenberg]{Ivo G.Rosenberg}
\address{Universit\'{e} de Montr\'{e}al,
C.P.6128,Succ.``Centre-ville'',
Montr\'{e}al P.Q. H3C 3J7,  Canada and Department of Mathematics and Statistics, Masaryk University, Brno, Czeck Republic,}
\email{rosenb@DMS.Umontreal.ca}

\date{\today}


\begin{document}

\maketitle

To the memory of Michel-Marie Deza,  with affection and admiration.

\begin{abstract} Let $A$ be an ordered alphabet, $A^{\ast}$ be the free monoid over $A$ ordered by   the Higman ordering, and let $F(A^{\ast})$ be the set of  final segments of $A^{\ast}$. With the operation of concatenation, this set is a monoid. We show that the submonoid $F^{\circ}(A^{\ast}):= F(A^{\ast})\setminus \{\emptyset\}$  is  free.  The MacNeille completion $N(A^{\ast})$ of  $A^{\ast}$ is a submonoid of $F(A^{\ast})$. As a corollary, we obtain that the monoid $N^{\circ}(A^{\ast}):=N(A^{\ast})\setminus \{\emptyset\}$ is free. We give an interpretation of the freeness of $F^{\circ}(A^{\ast})$ in the category of metric spaces over the Heyting algebra $V:= F(A^{\ast})$,  with the non-expansive mappings as morphisms. Each final segment  of $A^{\ast}$ yields the injective envelope $\mathcal S_F$  of a two-element metric space  over $V$.  The uniqueness of the decomposition of   $F$ is due to the uniqueness of the block decomposition of the graph $\mathcal {G}_{F}$ associated to this injective envelope. \end{abstract}

\section{Introduction and presentation of the main results}

The original motivation of this  paper is the work of Quilliot \cite{Qu1,Qu2}. He considers reflexive and directed graphs as metric spaces; the distance between two vertices $x$ and $y$ of a graph  $G$ being,  instead of a non-negative real, the set $d_G(x,y)$ of words over a two-letters alphabet $\{+, -\}$ which code the zig-zag paths going from $x$ to $y$. Then, he  uses concepts  of the theory of metric spaces like  \emph{balls}, \emph{non-expansive   maps},  
  and \emph{Helly property}. This point of view was extended to transition systems in \cite{pouzet-rosenberg}.   Indeed, one may view the graph $G$  as a transition system  $M$ over $\{+, -\}$ and the distance  as the language $d_M(x,y)$ accepted by the automaton $(M,\{x\}, \{y\})$ with initial state $x$ and final state $y$. In the case of reflexive and directed graphs, the values of the distance are final segments of the free monoid  $\{+, -\}^{\ast}$ equipped with the Higman ordering. To make the study of transitions systems over an alphabet $A$ closer to the graph case, it is convenient to suppose that the value of $d_M(x,y)$ determines the value of $d_M(y,x)$; for that, we suppose that the alphabet $A$ is equipped with an involution $-$ and each  transition system   $M:= (Q, T)$ is \emph{involutive}, in the sense that $(p,a, q)\in T$ if and only if $(q, \overline{a}, p)\in T$. Once the involution on $A$ is extended to the free monoid $A^{\ast}$ and then to the power set $\powerset({A^{\ast}})$, we have $d_M(x,y)= \overline {d_M(y,x)}$. Going a step further, we say that $M$ is \emph{reflexive} if  every letter occurs to every vertex, that is $(p, a, p)\in T$ for every $p\in Q$ and $a \in A$. In this case,  distances values are final segments of the free monoid $A^{\ast}$ equipped with the Higman ordering. 
 
 Structural properties of transition systems
rely upon algebraic properties of languages and conversely. In
fact, transition systems can be viewed as geometric objects
interpreting these algebraic properties. This paper is an illustration of this  claim.

We start with an  ordered alphabet $A$. Let $A^{\ast}$ be the free monoid equipped with the Higman ordering. Let $F(A^{\ast})$ be the set of final segments of $A^{\ast}$. The  concatenation of  words extends to $\powerset (A^{\ast})$; this operation defined by $XY:= \{\alpha\beta: \alpha\in X, \beta\in Y\}$ induces an operation on $F(A^{\ast})$  for which the set $A^{\ast}$ is neutral. Hence $F(A^{\ast})$ is a monoid.  Since it contains the empty set $\emptyset$ and $\emptyset$ has several decompositions (e.g. $\emptyset=\emptyset A^{\ast}= A^{\ast}\emptyset $), this monoid is not free.  Let  $F^{\circ}(A^{\ast}):= F(A^{\ast})\setminus \{\emptyset \}$ be the set of non-empty final segments of $A^{\ast}$.  This is submonoid of $F(A^{\ast})$ (see Subsection \ref{subsection-final} for definitions,  if needed). 

\begin{theorem}\label{prop1}
 $F^{\circ}(A^{\ast})$ is a free monoid.
 \end{theorem}
The existence (or not) of an involution on $A$ has no effet on the conclusion.
 
 The following  illustration of Theorem \ref{prop1} was proposed to us by J.Sakarovitch \cite{sakarovitch1}. An \emph{antichain} of  $A^{\ast}$ is any subset $X$ of $A^{\ast}$ such that any two distinct elements $\alpha$ and $\beta$ of $X$ are incomparable w.r.t. the Higman ordering. The set  $Ant (A^{\ast})$ of antichains of $A^{\ast}$ and the set $Ant_{<\omega}(A^{\ast})$ of finite antichains of $A^{\ast}$ are  submonoids of $\powerset( A^\ast)$; the sets $Ant^{\circ} (A^{\ast}):= Ant (A^{\ast})\setminus \{\emptyset\}$ and 
 $Ant_{<\omega}^{\circ}(A^{\ast}):= Ant_{<\omega}(A^{\ast})\setminus \{\emptyset\}$ of non-empty antichains are also submonoids. 
 From Theorem \ref{prop1}, we deduce:
 \begin{theorem} \label{antichainfree}
 The monoids  $Ant^{\circ} (A^{\ast})$ and $Ant^{\circ}_{<\omega}(A^{\ast})$  are free. 
 \end{theorem}
 
 Note that if $A$ is \emph{well-quasi-ordered} (w.q.o)(that is to say that every 
  final segment of $A$ is finitely generated)  then the  monoids $Ant (A^{\ast})$ and $Ant_{<\omega}(A^{\ast})$ are equal and isomorphic to the monoid $F(A^{\ast})$, thus Theorem  \ref{antichainfree} reduces to  Theorem \ref{prop1}. Indeed, if $A$ is w.q.o. then, according to a famous result of Higman \cite{higman}, $A^{\ast}$ is w.q.o. too,  that is  every final segment $F$ of $A^{*}$ is generated by $Min (F)$ the set of minimal elements of $F$. Since $Min(F)$ is an antichain and in this case a finite one, our claim follows.

 Let $N(A^{\ast})$ be the MacNeille completion of the poset $A^{\ast}$, that we may view as the collection of intersections of principal final segments of $A^{\ast}$. The MacNeille completion of $N(A^{\ast})$ is a submonoid of $F(A^{\ast})$. From Theorem \ref{prop1}, we derive:

\begin{theorem}\label{thm:completionfree} Let $A$ be an ordered alphabet. The monoid $N^{\circ}(A^{\ast}):=  N(A^{\ast})\setminus \{\emptyset\}$ is free.
\end{theorem}
We recall that a member $F$ of $F(A^{\ast})$ is \emph{irreducible} if it is distinct from $A^{\ast}$ and  is not the concatenation of two members  of $F(A^{\ast})$ distinct of $F$ (note that with this definition, the empty set is irreducible).  The fact that $F^{\circ}(A^{\ast})$ is free amounts to the fact that each  member decomposes in a unique way as a concatenation of  irreducible elements.  We interpret this fact by means of injective envelopes of $2$-element metric spaces. 

We suppose that $A$ equipped with an involution (this is not a restriction: we may choose the identity on $A$ as our involution). The category of metric spaces over $F(A^{\ast})$, with the non-expansive maps as morphisms has enough injectives (meaning that every metric space extends isometrically to an injective one).  The gluing of two injectives by a common vertex  yields an injective (see Theorem \ref{prop: collage}); we will  say that an injective which is not the gluing of two proper injectives is \emph{irreducible}. For every  final segment $F$ of $A^{\ast}$, 
 the $2$-element space metric space $E:= (\left\{
x,y\right\}, d)$ such that $d(x,y)=F$,  has an \emph{injective envelope} $\mathcal S_F$ (a minimal extension to an injective metric space).   To $\mathcal S_F$ corresponds a  transition system  $\mathcal {M}_{F}$ on the alphabet $A$, with transitions $(p, a, q)$ if $a \in d(p,q)$. The automaton  $\mathcal{A}_{F}:=  ({\mathcal M}_{F},\left\{ x\right\}, \left\{
y\right\} )$ with $x$ as initial state and $y$ as final state accepts $F$.  A transition system yields a directed graph  whose arcs are the ordered pairs $(x,y)$ linked by a transition. Since   the transition system $\mathcal {M}_{F}$ is reflexive and involutive and thus  the corresponding  graph $\mathcal {G}_{F}$ is undirected and has  a loop  at every vertex.  For an example, if $F= A^{\ast}$, $\mathcal S_F$ is the one-element metric space and $\mathcal {G}_{F}$ reduces to a loop. If $F= \emptyset$, $\mathcal S_F$ is the two-elements  metric space $E:= (\{x, y\}, d)$ with $d(x,y)= \emptyset$ and $\mathcal {G}_{F}$ has no edge.

With the notion of cut vertex and block borrowed from graph theory, we prove:

\begin{theorem}\label{thm:blockdecomposition}
Let $F$ be a  final segment of $A^{\ast}$ distinct of $A^{\ast}$. Then $F$ is irreducible if and only if  $\mathcal S_F$ is irreducible if and only if $\mathcal {G}_F$ has no cut vertex. If $F$ is not irreducible, 
the blocks of $\mathcal {G}_F$ are  the vertices of a finite path $C_0, \dots, C_{n-1}$ with $n\geq 2$,  whose  end vertices  $C_0$ and $C_{n-1}$ contain respectively  the initial state $x$ and the final state $y$ of the automaton $\mathcal{A}_{F}$ accepting $F$. Furthermore, $F=F_0 \dots  F_i \dots  F_{n-1}$, the automaton $\mathcal{A}_{F_i}$ accepting $F_i$  being isomorphic to  $({\mathcal M}_{F}\restriction C_{i},\left\{ x_i\right\}, \left\{
y_i\right\} )$, where $x_i:=x$ if $i=0$, $y_{i}= y$ if $i=n-1$ and  $\{x_i\}=C_{i-1} \cap C_{i}$, $\{y_i\}= C_{i} \cap C_{i+1}$, otherwise. 
\end{theorem} 

From this result, the freeness of $F^{\circ}(A^{\ast})$ follows.

%
%

An approach of transition systems as metric spaces was developped in \cite{JaMiPo,  KP1, kabil, KP2, pouzet-rosenberg, Sa}. A study of retraction, coretraction and injective objects among transition systems was also developped by Hudry \cite{hudry1, hudry2}.

This paper is organized as follows. The proof of Theorems \ref{prop1} and \ref{antichainfree} is in Section 2. The proof of Theorem \ref{thm:completionfree} is in Section 3.  Properties of
metric spaces over a Heyting algebra and their injective envelopes
are summarized in subsection 4.1. Involutive and reflexive transition systems  are presented in subsection 4.2. The injective envelope of a 
$2$-element metric space over $F(A^{\ast })$ is described in subsection  4.3.  We prove Theorem \ref {thm:blockdecomposition}  in subsection \ref{subsection:graph}. 

Part of these results  have been presented at the International Conference on Discrete Mathematics and Computer Science (DIMACOS'11) organized by A.~Boussa\"{\i}ri, M.~Kabil, and A.~Taik  in Mohammedia (Morocco) May, 5-8, 2011; a part of it was included  into the Th\`ese d'\'Etat defended by the first author \cite{kabil}. This paper benefited from discussions with several colleagues. The second author thanks Maurice Nivat for his support over the years on this theme.  We are particularly grateful for the encouragements of J.Sakarovitch.  

\section{The ordered monoids  $F(A^{\ast})$ and $Ant(A^{\ast})$}
In this section we prove Theorems \ref{prop1} \and \ref{antichainfree}. The proof of Theorem \ref{prop1} relies  on Levi's Lemma and a decomposition property we introduce at this occasion. The proof of Theorem \ref{antichainfree} is a consequence. 
\subsection {Monoids, Ordered monoids, Heyting algebras}\label{subsection:monoid}
Let $V$ be a monoid. We denote $\cdot$ the operation and ${\bf 1}$ its
neutral element.
The monoid $V$ is \textit{cancellative} if it is cancellative on the left and on the right that is if for all $u,v,w\in V$:
\begin{equation}\label{canc1}w\cdot u=w\cdot v\Longrightarrow u=v
\end{equation} and 
\begin{equation}\label{canc2} u\cdot w=v\cdot w\Longrightarrow u =v.
\end{equation}

The monoid  $V$  is \textit{equidivisible} if for all $u_1,u_2, v_1,v_2 \in V$:
 \begin{equation}\label{equ:equid}
u_1\cdot u_2=v_1\cdot v_2
\text{implies either}\;  u_1
=v_1\cdot w, v_2=w\cdot u_2\;  \text{or}\; v_1=u_1\cdot w,u_2=w\cdot v_2
\;  \text{for some}\; 
w\in V.
\end{equation}

We say that $V$ is \emph{graded} if there is a  morphism $\gamma$ of $V$ into the additive monoid of non-negative integers such that $\gamma^{-1}(0)=1.$ Such morphism $f$ is called a \emph{graduation}. 
We will use the following form of Levi's lemma  \cite{levi}  (cf \cite{lothaire} p.13, 1.1.1, Section 1.1, Problems).

\begin{lemma}\label{levi-lemma}
A monoid $V$ is free if and only if $V$ is equidivisible and graded. \end{lemma}

An element  $x$ of a monoid $V$ is \textit{irreducible} if  
\begin{equation}\label{equ:irr2}
 x\neq {\bf 1}\ \mbox{and}\ x = y \cdot z\Longrightarrow x=y\ \mbox{or}\ x=z. 
\end{equation}
We recall that
\begin{fact}\label{fact2}Every element of a monoid $V$ has a decomposition into irreducible elements provided that $V$  is graded. 
\end{fact}
The proof of this fact follows the lines  of E.Nether's proof that ideals of an Noetherian ring decompose into irreducible ideals. If this fact was not true, the subset $B$ of $x\in V$ with no decomposition into irreducible elements will be non empty. Pick $x\in B$ such that $\gamma(x)$ is minimal w.r.t the graduation $\gamma$. Check that  $x$ is irreducible. This yields a contradiction. 

 We denote by $Irr(V)$ the set of irreducible members of a monoid $V$. 
 
 \begin{lemma}\label{lem:free1} The  submonoid $W$ generated by some set $I$ of irreducible members  of a free monoid $V$ is free.   \end{lemma}
 
 Indeed, each element of $W$ has a unique decomposition as a product of members of $I$. 
 
An \emph{ordered monoid} is a monoid equipped with a compatible ordering. The ordered monoid is a \emph{meet-semilattice monoid} if the ordering is a meet-semilattice, that is every pair of elements $u, v\in V$ has a meet, denoted by $u\wedge v$, and if the monoid operation distributes with the meet, that is: 

\begin{equation}\label{equation:distrib1}
(w\cdot u) \wedge (w \cdot v)= w\cdot (u\wedge v)
\end{equation} 
\noindent and 
\begin{equation}\label{equation:distrib2}
(u\cdot w) \wedge (v  \cdot w)=(u\wedge v)\cdot w.
\end{equation}

The free monoid $A^{*}$ with the Higman ordering satisfies the following two conditions:
 \begin{equation}\label{eq:wlequv}
w\leq u\cdot v\;   \text{implies}\;  w=u'\cdot v'\;  \text{for some}\;  u'\leq u  \; \text{and}\; v'\leq v 
 \end{equation}
 and 
 \begin{equation}\label {eq:uvleqw}
  u\cdot v\leq w \; \text{implies}\;   w=u'\cdot v'\;  \text{for some}\; u\leq u'  \text{and}\; v\leq v'.
\end{equation}
 for all $u,v, w\in A^{\ast}$. 

\begin{lemma}\label{lem:subcancel} Let $V$ be an ordered monoid and $u,v, u', v'\in V$. 
\begin{enumerate}[{(a)}]

\item If $V$ is  cancellative  then $u\cdot v=u'\cdot
 v'$, $u\leq u'$ and $v\leq v'$ imply $u=u'$ and $v=v'$.
 
 \item If the neutral element ${\bf 1}$ is the least element of $V$, if $V$ is equidivisible    and satisfies  Condition (\ref{eq:wlequv}) or Condition (\ref {eq:uvleqw})  then $u\cdot v\leq u'\cdot
 v'$, resp. $u\cdot v<u'\cdot v'$,  implies $u\leq u'$, resp. $u<u'$, or  $v\leq v'$, resp. $v<v'$. 
 
 \end{enumerate} \end{lemma}
\begin{proof}$(a)$. Since $u\leq u'$ and $v\leq v'$ we have $u\cdot v\leq u'\cdot v\leq u'\cdot v'$. Since $u\cdot v=u'\cdot v'$ we have $u\cdot v= u'\cdot  v$. Since $V$ is cancellative, this yields $u=u'$. Similarly, we get $v=v'$.

$(b)$.  Suppose that $V$ satisfies Condition (\ref{eq:wlequv}). Suppose $u\cdot v\leq u'\cdot v'$. There are $u''$ and $v''$ such that $u''\leq u'$, $v''\leq v'$ and $u\cdot v= u''\cdot v''$. By equivisibility, either $u''= u\cdot u_1$ for some $u_1$, hence $u \leq u''$ in which case $u\leq u'$ or $v''=  v_1\cdot v$ for some $v_1$ hence  $v \leq v''$ in which case $v\leq v'$. If $u\cdot v< u'\cdot v'$ we get either $u<u'$ or $v<v'$. We get the same conclusion if $V$ satisfies Condition (\ref{eq:uvleqw}).
\end{proof} 
\begin{definitions}Let $V$ be an ordered monoid. The cartesian product $V\times V$ is ordered so that $(u_1,u_2)\leq (v_1,v_2)$ if $u_1\leq v_1$ and $u_2\leq v_2$. Let  $(v_1, v_2)\in V\times V$. This pair is \emph{above}  $v\in V$ if $v_1\cdot v_2 \geq v$. It is \emph {minimal above} $v$ if it is above $v$ and there is  no pair $(u_1,u_2)<(v_1, v_2) $ which is above $v$. It is \emph{minimal} if it is minimal above $v:= v_1\cdot v_2$.  The pair $(v_1,v_2)$ satisfies the \emph{convexity property} if 
for every minimal pair $(u_{1},u_{2})\in V\times V$ above  $v:= v_1\cdot v_2$
either:
\begin{equation}\label{equation:3}
 v_{1}\leq u_{1}\cdot u^{1}_{2}, \; u_{1}\leq v_{1}\ \mbox{ and}\ \ u_{2}=u^{1}_{2}\cdot v_{2} \ \ \
\mbox{for some}  \  u^{1}_{2} \in V
\end{equation}

or:
\begin{equation}\label{equation:4}
v_{2}\leq u^2_{1}\cdot u_{2}, \; u_{2}\leq v_{2}\ \mbox{and}\ u_{1}=v_{1}\cdot u^{2}_{1} \  \mbox{for
some}  \ u^{2}_{1}\in V.
\end{equation}

This pair is \emph{summable}  if it is minimal and satisfies the convexity property. 

The ordered monoid $V$ satisfies the \emph{decomposition property}  if every pair is summable. 
\end{definitions}

\begin{lemma}\label{lem:dec-property}
If a meet-semilattice monoid $V$ satisfies the decomposition property, then it is cancellative and
equidivisible.
\end{lemma}

\begin{proof}
 Suppose that $V$ satisfies the decomposition property. Then according to our definition, each pair $(v_1, v_2)$ is summable hence minimal  above  $v:= v_1\cdot v_2$. This property implies that $V$ is cancellative. Indeed, let $u, v, w\in V$ such that $w\cdot u= w\cdot v$. Due to distributivity, we have $w\cdot u= (w\cdot u)\wedge (w \cdot v)= w\cdot (u\wedge v)$. By minimality of $(w, u)$ above $w\cdot u$ we have $u \wedge v=u$, hence $u\leq v$. The minimality   of $(w, v)$ above $w\cdot v$ yields  similarly $v\leq u$, hence $u=v$, proving that $V$ is  cancellative on the left. The proof that $V$  is cancellative on the right is similar. Hence  $V$ is cancellative.

Let $u_1,u_2$ and $v_1,v_2$ such that $u_1\cdot u_2= v_1\cdot v_2$. Set $v:= v_1\cdot v_2$. Since $u_1\cdot u_2= v_1 \cdot v_2$, $(u_1, u_2)$ is above $v$. Since $(u_1,u_2)$ is summable, it is minimal above $u_1\cdot u_2$, that is minimal above $v$. 
 Since $(v_1,v_2)$ is summable, it satisfies the decomposition property. Hence Condition (\ref{equation:3})   or Condition (\ref{equation:4}) holds. Suppose that Condition (\ref{equation:3}) holds. Let $u^{1}_{2} \in V$ such that  $v_{1}\leq u_{1}\cdot u^{1}_{2}, u_{1}\leq v_{1}\ \mbox{ and}\ \ u_{2}=u^{1}_{2}\cdot v_{2}$. Since $v_1\cdot v_2= u_{1}\cdot (u^{1}_{2}\cdot v_2)= (u_{1}\cdot  u^{1}_{2})\cdot v_2$  and $V$ is cancellative we have $v_1= u_1\cdot u^{1}_{2}$. With $w:=u^{1}_{2}$,  Condition (\ref{equ:equid}) holds. If Condition (\ref{equation:4}) holds, we get the same conclusion. 
 Hence $V$ is equidivisible. 
 \end{proof}
 
 An ordered monoid $V$ is a   \emph{Heyting algebra} if the ordering is complete (every subset has a meet  and a join) and the following distributivity condition holds: 

\begin{equation}\label{distribinfty}
\bigwedge  _{\alpha \in A, \beta \in B} u_\alpha \cdot  v_\beta =
\bigwedge _{\alpha \in A} u_\alpha  \cdot  \bigwedge _{\beta \in B} v_\beta
\end{equation}
for all $u_\alpha \in V$ $(\alpha \in A)$ and
$v_{\beta} \in V$ $(\beta \in B)$. 

A Heyting algebra $V$  is \emph{involutive} if there is an involution $-$ on $V$ which preserves the ordering and reverses the monoid operation (that is $\overline {u\cdot v}= \overline v\cdot \overline u$ for all $u$ and $v$ in $V$) in particular the involution preserves the neutral element of the monoid.

In a Heyting algebra, the least element is not necessarily the neutral element for the monoid operation (in the next section, the set $\powerset (A^{\ast})$ of langages over an alphabet $A$ provides such an example). However, in the Heyting algebras we work with, namely $F(A^{\ast})$ and $Ant(A^{\ast})$, the least element and the neutral element coincide.

\subsection {The monoid of final segments}\label{subsection-final} Let $A$ be a set. Considering $A$ as an \textit{alphabet} whose
members are \textit{letters}, we write a word $\alpha $ with a
mere juxtaposition of its letters as $\alpha
=a_{0}\dots a_{n-1}$ where $a_{i}$
are letters from $A$ for 0 $\leq i\leq n-1.$ The integer $n$ is the \textit{%
length} of the word $\alpha $ and we denote it $\left| \alpha \right| $.
Hence we identify letters with words of length 1. We denote by $\Box $ the
empty word, which is the unique  word of length zero. 
The \emph{concatenation} of two word $\alpha:= a_{0}\cdots a_{n-1}$ and $\beta:=b_{0}\cdots b_{m-1}$ is the word $\alpha \beta:=a_{0}\cdots a_{n-1}b_{0}\cdots b_{m-1}$. We denote by $A^{\ast }$ the set of all words on the alphabet $A$. Once equipped with the
concatenation of words, $A^{\ast }$ is a monoid, whose neutral element is the empty word, in fact $A^{\ast}$ is the \textit{free
monoid} on $A$.   A \emph{language} is any subset $X$ of $A^{\ast}$. We denote by $\powerset (A^{\ast})$ the set  of languages. We will use capital letters for languages. If $X, Y \in \powerset (A^{\ast})$ the \emph{concatenation} of $X$ and $Y$ is the set  $XY:= \{\alpha\beta: \alpha\in X, \beta\in Y\}$ (and we will use  $Xy$ and $xY$instead of $X\{y\}$ and $\{x\}Y$).  This  operation  extends the concatenation operation on $A^{\ast}$; with it, the set $\powerset (A^{\ast})$ is a monoid whose  neutral element is the set $\{ \Box \}$.  Ordered by inclusion, this is  (join) lattice ordered monoid. Indeed, concatenation distributes over arbitrary union, namely:

\begin{center}
$( \underset{i\in I}{\bigcup }X_{i})Y=\underset{i\in I}{ \bigcup }X_{i} Y.$
\end{center}
But concatenation does not distribute over intersection (for a simple example, let $A:= \{a,b,c\}$, $I:=\{1, 2\}$, $X_1:=\{ab\}$, $X_2:=\{a\}$, $Y:=\{c, bc\}$, then $\emptyset =(X_1\cap X_2) Y\not = X_1Y\cap X_2Y=\{abc\}$). 
Hence, ordered by reverse of the inclusion, the monoid $\powerset( A^{\ast
})$ becomes a  Heyting algebra (while ordered by inclusion it is not). If $-$ is an involution on $A$, it extends  to an involution on $A^*$, by setting $\overline \Box:= \Box$, and $\overline{\alpha}=\overline {a_{n-1}}\dots\overline{a_0}$ if $\alpha= a_{0}\dots a_{n-1}$.  This  involution reverses the concatenation of words. Extended  to $\powerset (A^{\ast})$ by setting $\overline X:= \{\overline {\alpha }: \alpha\in X\}$, it reverses the concatenation of languages and preserves the inclusion order on languages.  
In summary:
\begin{lemma}\label{fact1:heyting}  The set $\powerset( A^{\ast
})$ equipped with the concatenation of languages and the reverse of the inclusion order is a Heyting algebra. Moreover, this is an involutive Heyting algebra if we add  to it  the extension of an involution on $A$. 
\end{lemma}

%
%
%
We suppose from now that the alphabet $A$ is
ordered. 
We
order  $A^{\ast }$ with the Higman ordering: if $\alpha $ and
$\beta $ are two elements in $A^{\ast }$ such $\alpha: =a_{0}\cdots 
a_{n-1}$ and $\beta: =b_{0}\cdots  b_{m-1}$ then $\alpha \leq \beta$
if there is an injective and increasing map $h$ from $\left\{
0,...,n-1\right\} $ to $\left\{ 0,...,m-1\right\}$ such that for each $i$,  $
0\leq i\leq n-1$, we have $a_{i}\leq b_{h\left( i\right) }$. Then
$A^{\ast }$ is an ordered monoid with respect to the concatenation
of words.  A \emph{final segment} of $A^{\ast}$ is any subset $F\subseteq A^{\ast}$ such that $\alpha \leq
\beta,\alpha \in F$ implies $\beta \in F$.  Initial segments are defined dually.  Let $X$ be a subset of 
$A^{\ast}$; then $$\uparrow X:= \{ \beta  \in A^{\ast}: \alpha \leq \beta\;  \text{for some}\;  \alpha\in X\}$$  is the {\it upper set} generated by $X$ and 
$$\downarrow X := \{\alpha  \in A^{\ast}:  \alpha \leq \beta \;  \text{for some}\;  \beta\in X\}$$   is the {\it lower set} generated by $X$. For a singleton $X = 
\{\alpha\}$, we omits the set brackets and call $\uparrow \alpha $ and $\downarrow 
\alpha $ a  \emph{principal upper set} and a \emph{principal lower set} respectively. 
 Let $F\left( A^{\ast }\right) $ be the collection of
final segments of $\ A^{\ast }$. 
The set $F\left( A^{\ast }\right)$ is stable w.r.t. the concatenation of languages:  if $ X,Y\in F\left(
A^{\ast }\right)$,  then $XY\in F(A^{\ast})$ (indeed, if $u,v, w\in A^{\ast}$ with $uv\leq w$ then $w= u'v'$ with $u\leq u'$ and $v\leq v'$).   Clearly,  the neutral element
 is $A^{\ast }$.  The set $F\left( A^{\ast }\right) $ ordered by
inclusion is a complete lattice (the join is the union, the meet
is the intersection).  Concatenation distributes over union. 
If we order $F\left( A^{\ast }\right) $ by reverse of the
inclusion, denoting $X\leq Y$ instead of $X\supseteq Y$, and we set ${\bf{1}}:= A^{\ast}$, we have 
%

\begin{lemma}\label{fact:heyting}The set $F( A^{\ast
})$ equipped with the concatenation of languages and the reverse of the inclusion order is a Heyting algebra. Moreover, this is an involutive Heyting algebra if we add  to it  the extension of an involution on $A$. 
\end{lemma}
A correspondance between $\powerset (A^{\ast})$ and $F(A^{\ast})$ is given in the following lemma.
\begin{lemma}\label{morphism1}The correspondance which associates  to every subset $X$ of $A^{\ast}$ the final segment $\uparrow X$ is a morphism of ordered monoids from $\powerset (A^{\ast})$ onto $F(A^{\ast})$. 
\end{lemma} 
\begin{proof} Clearly this correspondence preserves the ordering. Since by definition it is  surjective, to show that it is a morphism  of monoid it suffices to show that

\begin{equation}\label{eq:uparrow}
\uparrow X \uparrow Y= \uparrow (XY)
\end{equation}
for all $X, Y \in \powerset (A^{\ast})$. 

Let $z\in \uparrow X \uparrow Y$. Then  $z$ decomposes as $z= xy$ with  $x\in \uparrow X$ and $y\in \uparrow Y$. There are $x'\in X$ with $x'\leq x$ and $y'\in Y$ with $ 
y'\leq y$. Hence, $x'y'\in XY$ and $x'y'\leq xy=z$. Thus $z\in \uparrow (XY)$. This proves that $\uparrow X \uparrow Y\subseteq  \uparrow (XY)$. 

Conversely, let  $z\in \uparrow  (X Y)$. Then there are $x'\in X$ and $y'\in Y$ such that $x'y'\leq z$. Thus  $z$ decomposes as $z= xy$ with $x'\leq x$ and $y'\leq y$. Hence 
$x\in  \uparrow X$ and $y\in \uparrow Y$. Thus $z=xy\in  \uparrow X \uparrow Y$. This  proves that $\uparrow X Y\subseteq  \uparrow X\uparrow Y$. 
The equality holds, as claimed.
\end{proof}

An \emph{antichain} of  $A^{\ast}$ is any subset $X$ of $A^{\ast}$ such that any two distinct elements $x$ and $y$ of $X$ are incomparable w.r.t. the Higman ordering. Let  $Ant (A^{\ast})$ be the set of antichains of $A^{\ast}$ and $Ant_{<\omega} (A^{\ast})$ be the set of finite antichains. 

\begin{lemma} $Ant (A^{\ast})$ and $Ant_{<\omega} (A^{\ast})$ are submonoids of  $\powerset(A^{\ast})$. The morphism $X\mapsto \uparrow X$ from $\powerset(A^{\ast})$ into $F(A^{\ast})$  induces a one-to-one morphism from $Ant (A^{\ast})$ into $F(A^{\ast})$. The correspondance which associates  to every final segment $X$ of $A^{\ast}$ the set $Min(X)$ of its minimal elements is a morphism of monoids from $F(A^{\ast})$ onto $Ant(A^{\ast})$. 
\end{lemma} 
\begin{proof} We prove first  that: 

\begin{equation}\label{eq:min}
Min (X  Y)= Min (X) Min (Y)
\end{equation}
for all $X, Y \in F(A^{\ast})$. 

Let $z\in Min (XY)$. Since $z\in XY$,   it decomposes as $z=xy$ with $x\in X$ and $y\in Y$.  We prove that $x\in Min (X)$ and $y\in Min (Y)$, from which follows that $z=xy\in Min (X) Min (Y)$ and thus the inclusion $Min (X  Y)\subseteq Min (X) Min (Y)$. If $x\not \in Min (X)$ there is some $x' \in X$ with $x'<x$. In this case, $x'y<xy=z$. Since $x'y\in XY$ this contradicts the minimality of $z$. Thus $x\in Min (X)$. By the same argument, we have $y\in Min (Y)$. 

Let $z\in Min (X)Min (Y)$. We claim that $z\in Min(XY)$ from which the inclusion $Min (X)Min (Y)\subseteq Min(XY)$ follows. This element  $z$ decomposes as $z=xy$ with $x\in Min (X)$ and $y\in Min (Y)$. In particular, $z\in XY$. If $z\not \in Min (XY)$, there is some $z'\in XY$ with $z'< z$. This element $z'$ decomposes as $z'= x'y'$ with $x'\in X$, $y'\in Y$. Since $x'y'<xy$ then  according to $(b)$ of Lemma \ref{lem:subcancel} either $x'<x$ or $y'<y$. The first case is impossible since $x\in Min(X)$ and the second too since $y\in Min (Y)$. Thus $z\in Min(XY)$ as claimed. 

Let $U, V\in Ant(A^{\ast})$. Then $UV\in Ant(A^{\ast})$ and 

\begin{equation}\label{eq:antichains}
\vert UV\vert= \vert U \vert \vert V\vert. 
\end{equation}

 We have $U= Min(\uparrow U)$ and $V= Min (\uparrow V)$ and by Equation  (\ref{eq:min}) we have $UV= Min(\uparrow U\uparrow V)$, proving that $UV$ is an antichain. The map $(x,y) \mapsto xy$ is a bijection from $U\times V$ to $UV$. Indeed, if $xy=x'y'$ then by $(b)$ of Lemma \ref{lem:subcancel} either $x\leq x'$ or $y\leq y'$. If $x\leq x'$ then since $U$ is an antichain then $x=x'$; since $A^{\ast}$ is free it is cancellative,  hence $y=y'$. Similarly, the case $y'\leq y$ yields  $y=y'$ and $x=x'$.  Equation (\ref{eq:antichains}) follows.

 Concatenation preserves  $Ant(A^{\ast})$ and, according  to Equation (\ref{eq:antichains}), also $Ant_{<\omega}(A^{\ast})$. Since  for each $x\in A^{\ast}$, $\{x\}$ is an antichain, $\{\Box\}$ is an antichain. Since this is  the neutral element of $\powerset (A^{\ast})$,  this is the neutral element of $Ant(A^{\ast})$ and $Ant_{<\omega}(A^{\ast})$ which are then  submonoids of $\powerset (A^{\ast})$.

Since $Ant(A^{\ast})$ is a  submonoid of $\powerset (A^{\ast})$, the map $X\mapsto \uparrow X$ from $\powerset(A^{\ast})$ into $F(A^{\ast})$ induces  a morphism from $Ant (A^{\ast})$ into $F(A^{\ast})$. This morphism is one-to-one. Indeed, if $U$ is an antichain, $U= Min (\uparrow U)$. 
The map $X\mapsto Min (X)$ transforms the neutral element of the monoid $F( A^{\ast})$, namely $A^{\ast}$,   into $\{\Box\}$ which is the neutral element of the monoid $Ant(A^{\ast})$. Since, according to  Equation (\ref{eq:min}), this maps  preserves the concatenation, it is a morphism of monoid. 
\end{proof}

\begin{lemma} \label{fact:graded} The set $F^{\circ} (A^{\ast}):= F(A^{\ast})\setminus \{\emptyset \}$ is a  graded and cancellative submonoid  of $F\left( A^{\ast}\right)$.  
\end{lemma}
\begin{proof}
 Set $\gamma (X):= \Min \{ \vert x\vert : x\in X \}$ for every $X\in F^{\circ}(A^{\ast})$.  This is a graduation. 
 
 Let $X,Y,Z$ be three elements in $F^{\circ}( A^{\ast })$ Then

$$XY=XZ\Longrightarrow Y=Z   \eqno(1) $$
and 
$$ YX=ZX\Longrightarrow Y=Z.    \eqno(2) $$\\
Indeed, let $x\in X$ such that $\left| x\right| = \gamma(X)$.  Let $y\in Y.$ Since $xy\in XY$,  there exists $x^{\prime }\in X$ and  $z\in Z$ such that $xy=x^{\prime }z.$ From the equidivisibility
property of $A^{\ast}$ and the fact that $\left| x\right| \leq \left| x^{\prime
}\right|,$we have $z\leq y.$ Since $Z$ is a final segment, it
follows that $y\in Z.$ Hence $Y\subseteq Z.$ Similarly we get
$Z\subseteq Y$ proving $Y=Z$. By the same argument we prove
$\left( 2\right) .$  \end{proof}
Since $F^{\circ} (A^{\ast})$ is cancellative,  it  satisfies $(a)$ of Lemma \ref{lem:subcancel}, that is:
\begin{lemma}\label{lem:semicancel} Let $X,Y, X', Y'$ be non-empty final segments of $A^{\ast}$. 
If $XY=X'Y'$ with $X\subseteq X'$ and $Y\subseteq Y'$ then $X=X'$ and $Y=Y'$.
\end{lemma}
\begin{lemma}\label{lem:product}Let $Z$ be a non-empty  antichain. If $\uparrow Z= XY$ with $X, Y\in F(A^{\ast})$ then $X= \uparrow Min (X)$ and $Y= \uparrow Min (Y)$.  \end{lemma} 
\begin{proof}
We have  $Min (\uparrow Z) = Min (XY)= Min (X)Min (Y)$ by Equation (\ref{eq:min}).  With Equation (\ref {eq:uparrow}) this yields   $\uparrow Min (\uparrow Z)=  \uparrow (Min (X) Min (Y))=\uparrow  Min (X)\uparrow Min (Y)$.  Since $Z= Min (\uparrow Z)$, we have $\uparrow Z= \uparrow Min(X)\uparrow Min (Y)= XY$. Since $\uparrow Min (X) \subseteq X$ and $\uparrow Min (Y) \subseteq Y$, we have $\uparrow Min(X)=X$ and $\uparrow Min( Y)=Y$ by Lemma \ref{lem:semicancel}.  \end{proof}

\begin{corollary}\label{cor:free2} The one-to-one morphism $Z\mapsto \uparrow Z$ from $Ant(A^{\ast})$ into $F(A^{\ast})$ maps the irreducibles of $Ant(A^{\ast})$ and $Ant_{<\omega}(A^{\ast})$  into the irreducibles of $F(A^{\ast})$. 
\end{corollary}
\begin{proof} Let $Z$ be an irreducible of $Ant(A^{\ast})$. We claim that $\uparrow Z$ is irreducible in $F(A^{\ast})$. Indeed, suppose $\uparrow Z= XY$ with $X, Y\in F(A^{\ast})$. Since $Z$ is an antichain, $Z=Min (\uparrow Z)= Min (XY)= Min (X)Min(Y)$. Since $Z$ is irreducible in $Ant(A^{\ast})$, either  $Z= Min(X)$ or $Z= Min (Y)$. Hence, either $\uparrow Z= \uparrow Min(X)$ or $\uparrow Z= \uparrow Min (Y)$.  If $\uparrow Z= \emptyset$ then necessarily $X= \emptyset$ or $Y=\emptyset$, hence $\uparrow Z= X$ or $\uparrow Z=Y$ and thus $\uparrow Z$  is irreducible. If $\uparrow Z\not= \emptyset$, then  according to Lemma \ref {lem:product},  $X= \uparrow Min (X)$ and $Y= \uparrow Min (Y)$.  Thus, if  $\uparrow Z= \uparrow Min(X)$, $\uparrow Z= X$ and if   $\uparrow Z= \uparrow Min (Y)$, $\uparrow Z= Y$, proving that $\uparrow Z$ is irreducible.
\end{proof}

\begin{lemma}\label{decproperty}
The ordered monoid $F^{\circ}( A^{\ast })$ satisfies the
decomposition property. 
\end{lemma}

\begin{proof}
 According to the definition given Subsection \ref{subsection:monoid}, we neeed to prove that all pairs $(V_1, V_2)$ of  $ F^{\circ}( A^{\ast })$ are summable, that is are minimal and have the convexity property (minimality being w.r.t.  the reverse of inclusion).

$\bullet$  Minimality of  a pair $(V_1,V_2)$ means that: 
 $$V_{1}V_{2}=U_{1}U_{2}, \
\
V_{1}\subseteq U_{1}\;  \text{and}\; V_{2}\subseteq U_{2}\Longrightarrow V_{1}=U_{1} \; \text{and}  \; V_{2}=U_{2}. \eqno(3)$$ 
 for all  $U_{1},U_{2}\in F^{\circ}( A^{\ast }).$

This property readily follows from Lemma \ref{lem:semicancel}.


 $\bullet$ Convexity means that for all $U_{1},$ $U_{2} \in F^{\circ}( A^{\ast })
$ where  $(U_{1}, U_{2})$ is a minimal pair (with respect to reverse of inclusion) such that  $U_1U_2\subseteq
V:= V_{1}V_{2},$\\   we have
\noindent either:
\begin{equation}\label{equation:7}
  U_{1}U^{1}_{2}\subseteq  V_{1}, V_{1}\subseteq U_{1}  \ \mbox{ and}\ \ U_{2}=U^{1}_{2}V_{2} \ \ \
\mbox{for some}  \  U^{1}_{2}\in F^{\circ}( A^{\ast })\end{equation}
or
\begin{equation}\label{equation:8}
 U^2_{1}U_{2}\subseteq V_{2}, V_{2}\subseteq U_{2}\ \mbox{and}\
U_{1}=V_1U^{2}_{1} 	  \ \mbox{for some} 
\;  U^2_{1} \in F^{\circ}( A^{\ast }).
\end{equation}

Let $U_{1},U_{2}\in F^{\circ}( A^{\ast }) $ such that $(U_{1}, U_{2})$ is a minimal pair (with respect to reverse of inclusion) such that  $U_1U_2\subseteq
V:= V_{1}V_{2}$. 

\noindent{\bf Claim 1.} If $U_1\subseteq V_1$ and $U_2\subseteq V_2$ then  $U_1=V_1$ and $U_2=V_2$ hence both Conditions (\ref{equation:7}) and (\ref{equation:8}) hold. \\
This follows directly from the minimality of the pair $(U_1,U_2)$. 


\noindent {\bf Claim 2.} Either $%
U_{1}\subseteq V_{1}$ or $U_{2}\subseteq \ V_{2}$.\\
Indeed, If $U_{1}\nsubseteq V_{1}$ and $U_{2}\nsubseteq   V_{2},$ then let $%
u_1\in  U_{1}\setminus V_{1}$ and $u_2 \in U_{2}\setminus V_{2}.$ We
have $u_1u_2\in U_{1}U_{2}\subseteq V_{1}V_{2}.$ Let $v_{1}\in V_{1}, $
$v_{2}\in V_{2}$ such that $u_1u_2=v_{1}v_{2}.$ The equidivisibility
property of $A^{\ast}$ implies that either $v_{1}$ is a left factor of $u_1$ or $v_2$ is a right
factor of $u_{2}$. Therefore either $v_{1}\leq u_1$ or $v_{2}\leq
u_2$ Since $V_{1}$ and $V_{2}$ are final segments of $A^{\ast }$, \
this implies $u_1\in V_{1}$ or $u_2\in V_{2} $ which contradicts the
choice of $u_1$ and $u_2$.\\
{\bf Claim 3.}  If $U_{1}\subseteq V_{1}$ and $U_{2} \nsubseteq  
V_{2} $  then  Condition (\ref{equation:8}) holds. \\
Indeed, let $u\in U_{2}\setminus V_{2}$. We may write $U_{1}=\underset{i\in I}{%
\text{ }\bigcup }Z_{i}$ where for each $i\in I,Z_{i}$ is of the
form $\left\{ x\in A^{\ast }:\alpha _{i}\leq x\right\} $ for some
$\alpha _{i}\in U_{1}.$ Observe first, for every $i\in I,\alpha _{i}$ has a
proper left factor $\gamma _{i}\in V_{1}.$ Indeed, we have $\alpha _{i}u\in U_{1}U_{2}\subseteq V_{1}$ $V_{2}.$ There are $%
\gamma _{i}\in V_{1},\delta _{i}\in V_{2}$ such that $\alpha
_{i}u=\gamma _{i}\delta _{i}.$ From the equidivisibility property,
 either $\gamma _{i}$ is a left factor of $\alpha
_{i}$ or $\alpha _{i}$ is a left factor of $\ \gamma _{i}.$ The
later relation is impossible, since we would get $\delta _{i}\leq
u$ and then since $V_{2}$ is a final segment of $A^{\ast },$ this
would imply $u\in V_{2}.$ Since $A^{\ast }$ is cancellative, the
choice of $u$ ensures that $\gamma _{i}$ is a proper left factor
of $\alpha _{i}$. Let $i\in I$ and let $\beta _{i}$ be the least proper left factor  of $%
\alpha _{i}$ which belongs to $V_1$. Let $\xi _{i}\in A^{\ast }$such that $\alpha
_{i}=\beta _{i}$ $\xi _{i}.$ Set $U^{2}_{1}: =\underset{i\in I}{\text{
}\bigcup }W_{i}$ where $W_{i}:=\left\{ x\in A^{\ast }:\xi _{i}\leq
x\right\} $ and $U^{2}_{2}:=\{x\in A^{\ast }:$ for all $a\in U^{2}_{1}, ax\in
V_{2}\}.$ First,  we show that $U_{1}\subseteq V_{1}U^{2}_{1}.$ Let $z\in
U_{1}.$ There exists $i \in I$ such that $\alpha _{i}\leq z.$ Thus $\alpha _{i}=\beta _{i}$ $%
\xi _{i}\leq z.$ There are $\beta _{i}^{\prime },\xi _{i}^{\prime }\in A^{\ast }$ such that $%
z=\beta _{i}^{\prime }$ $\xi _{i}^{\prime },\beta _{i}\leq \beta
_{i}^{^{\prime }}$ and $\xi _{i}\leq \xi _{i}^{^{\prime }}.$ Since $V_{1}$ 
and $U^{2}_{1}$ are final segments, we have $\beta _{i}^{^{\prime }} \in
V_{1}$ and $\xi _{i}^{\prime }\in U^{2}_1$ proving $z\in V_{1}U^{2}_{1}.$ Next, 
we have trivially $U^2_{1}U^{2}_{2}\subseteq V_{2}.$ Finally, we prove that 
$U_{2}\subseteq U^{2}_{2}.$ Suppose that is false and let $u^{\prime }\in
U_{2}\setminus U^{2}_{2}.$ This means that there is $x\in U^{2}_{1}$ such that
$xu^{\prime }\notin V_{2}.$ Let $i\in I$ such that $\xi _{i}\leq x.$ Since $V_{2}$ is a final segment and $%
\xi _{i}u'\leq x _{i}u', $we get $\xi _{i}u'\notin V_{2}.$ But $\beta _{i}\xi _{i}u'=\alpha _{i}u'\in U_{1}U_{2} \subseteq V_{1} V_{2}.$ Let $z_{1}\in V_{1}$ and $z_{2}\in
V_{2}$ such that $\beta _{i}\xi _{i}u=z_{1}z_{2}.$ Since
$\xi _{i}u' \notin V_{2},$ it follows from the
equidivisibility property that $z_{1}$ is a proper left factor of
$\beta _{i}.$ This contradicts the choice of $\beta _{i}.$
In summary,  we have $U_{1}\subseteq V_{1}U^{2}_{1}$,   $U^2_{1} U^{2}_{2}\subseteq V_{2}$, $U_{2}\subseteq U^{2}_{2}$. With the minimality of the pair $(U_1,U_2)$,  we prove  $U_{1}=V_{1}U^{2}_{1}$ and $V_2\subseteq U_2$. Indeed, we have $U_{1}\subseteq V_{1}U^{2}_{1}\subseteq V_{1}$, hence $U_{1}\subseteq V_{1}$. With $ V_{1}V_{2}\subseteq V$ we obtain $ U_{1}V_{2}\subseteq V$. With $U_{1}U_{2}\subseteq V$ we have $ V\supseteq  U_{1}V_{2}\cup U_{1}U_{2} = U_{1}(V_{2}\cup U_{2})$. The minimality of $(U_1,U_2)$ implies $U_{2}=V_{2}\cup U_{2 }$, that is $V_{2}\subseteq U_{2}$. From $U^2_{1} U^{2}_{2}\subseteq V_{2}$, we get $V=V_{1}V_{2}\supseteq V_{1}U^2_{1} U^{2}_{2} $. Since $U_{2}\subseteq U^{2}_{2}$ and $U_{1}\subseteq V_{1}U^{2}_{1}$, minimality of $(U_1,U_2)$ yields $U^{2}_{2}=U_{2}$ and $V_{1}U^{2}_{1}=U_{1}$. From $V_{2}\supseteq U^2_{1} U^{2}_{2} $ and $U^{2}_{2}=U_{2}$ we have $ V_2\supseteq U^2_{1}U_2$. Hence,  Condition (\ref{equation:8}) holds.

Similarly we prove:

\noindent {\bf Claim 4.} $U_{1}\not \subseteq V_{1}$ and $U_{2}\subseteq V_{2}$   imply  that Condition (\ref{equation:7}) holds.

Convexity property follows from these claims.  
\end{proof}

\subsection{Proof of Theorem \ref{prop1}}
According to Lemma \ref{decproperty}, 
$F^{\circ}( A^{\ast })$ satisfies the
decomposition property.  From Lemma \ref{lem:dec-property} it is equidivisible. It is graded by Lemma \ref{fact:graded}. From Levi's Lemma (Lemma \ref{levi-lemma}) it is free. \hfill $\Box$
\subsection{Proof of Theorem \ref{antichainfree}} According to Corollary \ref{cor:free2},  the set  of  irreducibles of $Ant(A^{\ast})^{\circ}$ is mapped into the set of irreducibles of $F(A^{\ast})^{\circ}$. Since  the monoid $F^{\circ}( A^{\ast })$ is free, the monoid generated by this subset of irreducibles is free (Lemma \ref{lem:free1}). Hence  $Ant(A^{\ast})^{\circ}$ and $Ant_{<\omega}(A^{\ast})^{\circ}$   are free. 
\hfill $\Box$

\begin{remark}Let $Irr_A$ be the set of non-empty irreducible members of $F(A^{\ast})$. The monoid $F(A^{\ast})\setminus \{\emptyset\}$ is isomorphic to the monoid $(Irr_A)^*$. But, as an ordered monoid, $F^{\circ}(A^{\ast})$ is not isomorphic to the ordered monoid $(Irr_A)^{\ast}$ equipped with the Higman ordering w.r.t the reverse of inclusion on $Irr_A$.  Indeed, let $A:= \{a,b\}$,  $X:= \uparrow \{a,b\}= A^*\setminus \{\Box\}$ and $Y:= \uparrow \{aa,bb\}$. Then $X$ and $Y$ are irreducible and $XX\leq Y$. \end{remark}

\begin{problem}
Describe the irreducible of $F(A^{\ast})$
\end{problem} 

For an example, if  $F= \uparrow \{u, v\}$ with  $u$  incomparable to $v$, then $F$ is  irreductible iff  $u$ and $v$ do not have a common prefix nor a common suffix.


\section{The MacNeille completion of the free monoid} 
The ordered monoid $A^*$ can be extended to a complete lattice ordered monoid by applying the MacNeille completion.The necessary notation (cf. Skornjakow \cite{Sko}, 1973), Lemma  \ref{lem1} and Theorem \ref{thm1:bandelt-pouzet} of \cite {bandelt-pouzet} are  introduced next.

 Let $X$ be a subset of $A^{\ast}$. Then
$$X^{\Delta}:= \bigcap_{x \in X} \uparrow x$$
 and 
$$X^{\nabla}:= \bigcap_{x \in X} \downarrow x$$
are the {\it upper cone} and  the {\it lower cone} respectively,  generated by $X$. 

The pair $(\Delta, \nabla)$ of mappings on $\powerset(A^{\ast})$,  the power set lattice of $A^{\ast}$, 
constitutes a Galois connection, yieldings the \emph{MacNeille completion} of $A^{\ast}$. This completion is realized  as the complete 
lattice $$\{W \subseteq A^{\ast} : W = W^{\Delta \nabla}\}$$ 
ordered by inclusion or its isomorphic copy  $$\{Y \subseteq A^{\ast}: Y = 
Y^{\nabla \Delta}\}$$ ordered by reverse inclusion. The set $A^{\ast}$ 
embeds into the former via $x \mapsto \downarrow x$ and into the 
latter via $x \mapsto \uparrow x$ ($x\in A^{\ast}$).

The completion of $A^{\ast}$ inherits its monoid structure from the power set. The cone operators preserve this multiplication as the following lemma confirms.

For reader convenience, we give the proof. 
\begin{lemma}\label {lem1}
For any subsets  $X,Y$ of  $A^{\ast}$,   
$$(XY)^{\nabla} = X^{\nabla}Y^{\nabla}, \ {\rm and}\ (XY)^{\Delta} = 
X^{\Delta} Y^{\Delta}\;  \text {if}\;  X, Y\not =\emptyset, $$
hence
$$(XY)^{\nabla \Delta} = X^{\nabla \Delta}Y^{\nabla 
\Delta}\ {\rm and}\ (XY)^{\Delta \nabla} = X^{\Delta \nabla}
Y^{\Delta \nabla}.$$
\end{lemma}

\begin{proof}

First, observe that $\emptyset^{\nabla}= \emptyset ^{\Delta}= A^{\ast}$ and $(A^{\ast})^{\Delta}=\emptyset$, while $(A^{\ast})^{\nabla}:= \{\Box\}$. Further, $\emptyset Z=\emptyset$ for every subset $Z$ of $A^{\ast}$. The inclusions $X^{\nabla} Y^{\nabla} \subseteq (XY)^{\nabla}$ and
$X^{\Delta} Y^{\Delta} \subseteq (XY)^{\Delta}$ are thus immediate.

Suppose that there exists a word $w$ in $(XY)^{\nabla}$ that does not 
belong to $X^{\nabla} Y^{\nabla}$.  Then let $u$ be the longest prefix 
of $w$ from $X^{\nabla}$, and let $v$ be the longest suffix 
of $w$ from   $Y^{\nabla}$ so that  $w$  is of the form
$$w =u {a_{1}} \ldots a_{k} v$$
for some letters $a_{1}, \ldots, a_{k}$, where $k \geq 1.$ 
By the choice  of $u$ and $v$, there are words $x \in X$ and $y \in Y$ 
such that
$$u a_{1} \not \leq x\ {\rm and}\ a_{k} v \not \leq y.$$
This, however, is in conflict with
$$w = u a_{1} \ldots a_{k} v \leq xy.$$
Therefore $(XY)^{\nabla}$ equals $X^{\nabla} Y^{\nabla}.$
Finally, suppose  $z$ is a word in $(XY)^{\Delta}$ which does not 
belong to $X^{\Delta} Y^{\Delta}$, where $X$ and $Y$ are nonempty. Then the shortest  prefix 
of $z$ from $X^{\Delta}$ and the shortest suffix of $z$ from 
$Y^{\Delta}$ intersect in a nonempty subword
$$w:=a_{1} \ldots a_{k}$$ 
\noindent so that $z$ can be written as 
$$z = uwv\ {\rm with}\ uw \in X^{\Delta}\ {\rm and}\ wv \in 
Y^{\Delta}.$$
By the choice of the words $u$ and $v$, we can find words $x \in X$ and $y \in 
Y$ with  
$$x \not \leq u a_{1} \ldots a_{k-1}\ {\rm and}\ y \not \leq 
v.$$
This contradicts the hypothesis that 
$$xy \leq z = u{a_{1}} \ldots a_{k-1} a_{k} v.$$
We conclude that $(XY)^{\Delta} = X^{\Delta} Y^{\Delta}$, completing the proof.  
\end{proof}

The completion of $A^{\ast}$, realized by the upper closed sets, that we denote by $N(A^{\ast})$,  is a complete lattice in which suprema are  set-theoretic intersections, whereas infima are the closures of set-theoretic unions. The   \emph{closed union} of a family $Z_i (i\in I)$ of upper sets in $A^{\ast}$ is given by:
$$ \bigsqcup_{i\in I}Z_i= (\bigcup_{i\in I} Z_i)^{\nabla\Delta}.$$

The following result entails that the completion of $A^{\ast}$ is a  complete latttice ordered monoid (in the sense of Birkhoff, 1967 \cite{birkhoff}). 
\begin{theorem}\label{thm1:bandelt-pouzet}
For any ordered alphabet $A$, the collection  $N(A^{\ast})$ of all closed upper sets of 
words over $A^{\ast}$ is a monoid and   complete lattice such that 
the multiplication distributes over intersection  and closed unions, that is :
$$Y(\bigcap_{i \in I} Z_{i}) = \bigcap_{i \in I} YZ_{i}\ {\rm and}\ 
(\bigcap_{i \in I} Z_{i}) Y = \bigcap_{i \in I} Z_{i} Y,$$
$$Y(\bigsqcup_{i \in I} Z_{i}) = \bigsqcup_{i \in I} YZ_{i}\ {\rm and}\ 
(\bigsqcup_{i \in I} Z_{i}) Y = \bigsqcup_{i \in I} Z_{i}Y$$
for any index set $I$ and all closed upper sets $Y, Z_{i} (i \in I).$
\end{theorem}

According to Theorem \ref{thm1:bandelt-pouzet}, $N(A^{\ast})$ is a submonoid of $F(A^{\ast})$ and a Heyting algebra too. Also, $N^{\circ}(A^{\ast}):=N(A^{\ast})\setminus \{\emptyset\}$ is a submonoid of   $F^{\circ}(A^{\ast})= F(A^{\ast})\setminus \{\emptyset\}$. 
\subsection{\noindent \bf Proof of Theorem \ref{thm:completionfree}} 

\begin{lemma}Let $Z\in N^{\circ}(A^{\ast})$. If $Z= XY$ with $X, Y\in F(A^{\ast})$ then $X, Y\in  
N^{\circ}(A^{\ast})$. 
\end{lemma}
\begin{proof} According to Lemma \ref{lem1}, 
$(XY)^{\nabla} = X^{\nabla}Y^{\nabla}$ and  
$(XY)^{\nabla \Delta} = X^{\nabla \Delta}Y^{\nabla 
\Delta}$. Since $Z=XY$ and  $Z^{\nabla \Delta}=Z$ we have $Z=  X^{\nabla \Delta}Y^{\nabla 
\Delta}$. Since $X\subseteq X^{\nabla \Delta}$,   $Y \subseteq Y^{\nabla \Delta}$
and $Z= XY$, we have $X=X^{\nabla \Delta}$ and   $Y= Y^{\nabla \Delta}$ by Lemma \ref{lem:semicancel}.  
\end{proof}
 From this lemma, the irreducible members of $N^{\circ}(A^{\ast})$ are irreducible in $F^{\circ}(A^{\ast})$. 
 According to Lemma \ref {lem:free1}, $N^{\circ}(A^{\ast})$ is free.

\section{Metric spaces over $F(A^{\ast})$}
\subsection{Basics on metric spaces over a Heyting algebra}

  The following is a brief outline of \cite {KP2}. Let $V$ be
an ordered monoid, the operation being denoted multiplicatively, the neutral element denoted by ${\bf 1}$ (denoted respectively $+$ and $0$ in \cite {KP2}). We suppose $V$ equipped  
 with an involution $-$ such that $\overline {u\cdot v}= \overline v\cdot \overline u$ for every $u,v \in V$. Let $E$ be a set. A \textit{$V$-distance} on $E$ is a map $%
d:E^{2}\longrightarrow V$ satisfying the following properties for all $%
x,y,z\in E:$

d1) $\ d(x,$ $y)={\bf1}\Longleftrightarrow x=y$, 

d2) $\ d(x,$ $y)\leq d\left( x,z\right) \cdot d\left( z,y\right) $, 

d3) $d(x,$ $y)=\overline{d\left( y,x\right) }.$

The pair $\left( E,d\right) $ is called a $V$-\emph{metric space}. If there is no
danger of confusion we will denote it by $E.$ 
These notions appear  in \cite{deza-deza} (cf. p.41) under the name of \emph{generalized metric} and \emph{generalized distance space} (with the difference that the law is denoted additively and ${\bf 1}$ is replaced by $0$).\\ 

If $V$ is a Heyting algebra (i.e. satisfies the distributivity condition given in equation (\ref{distribinfty})), a $V$-distance can be defined on $%
V$. This fact relies on the classical notion of \textit{residuation}. Let $v\in V$. Given $\beta \in V$, the sets
 $\{r \in V: v \leq r \cdot \beta\}$ and
$\{r \in V: v \leq  \beta \cdot r \}$ have least elements, that we 
denote
respectively by $\lceil v\cdot \beta^{-1}\rceil$ and $\lceil \beta^{-1} \cdot  v  \rceil$
(note that  $\overline {\lceil \beta^{-1}\cdot v \rceil} =
\lceil \bar v\cdot (\bar \beta)^{-1} \rceil$). It follows that for all
$p, q \in V$, the set
$$D(p,q):=\{r \in V : p\leq q \cdot  \bar r\;\;{\rm and}\; \; q\leq 
p\cdot r\}$$
has a least
element, namely $\lceil \bar p\cdot (\bar q)^{-1}) \rceil \vee
\lceil p^{-1}\cdot q \rceil$, that we denote by $d_V(p,q)$. As shown in \cite{JaMiPo}, the
map $(p,q) \longrightarrow d_{V}(p,q)$
is a $ V-$distance.

Let $(E,d)$ and $\left( E^{\prime },d^{\prime }\right) $ be two
$V-$metric spaces. Recall that a map $f:E\longrightarrow E^{\prime }$ is a
\emph{non-expansive map} (or a \emph{contraction}) from $(E$, $d)$ to $\left( E^{\prime },d^{\prime
}\right) $ provided that $d^{\prime }(f(x),(f(y))\leq d(x,y)$ holds for all $%
x,y,\in E$.  The map $f$ is an \textit{isometry} if $d^{\prime
}(f(x),(f(y))=d(x,y)$ for all $x,y,\in E$. We say that $E$ and
$E^{\prime }$ are \textit{isomorphic}, a fact that we denote by
$E\cong E^{\prime }$, if there is a surjective isometry from $E$
onto $E'$.

Let $\left( (E_{i}, d_{i})\right) _{i\in I}$ be a family of $V$-metric spaces. The \emph{direct product}  $\underset{i\in I}{%
\prod }\left( E_{i}, d_{i}\right) $, is the metric space $(E,d) $ where $E$ is the cartesian product  $%
\underset{i\in I}{\prod }E_{i}$ and $d$ is the  ''sup'' (or
$\ell ^{\infty }$) distance  defined
by $d\left(
\left( x_{i}\right) _{i\in I},\left( y_{i}\right) _{i\in I}\right) =%
\underset{i\in I}{\bigvee }d_{i}(x_{i},$ $y_{i})$.

For a $V$-metric space $E$, let $x$ $\in E$ and $r\in V$, we define the {\it ball} $B_{E}(x,r)$ as the set \{$%
y\in E:d\left( x,y\right) \leq r\}$. 
We say that  $E$ is \emph{convex} if the intersection of two balls $B_{E}(x_{1}$, $r_{1})$ and $B_{E}(x_{2}$, $r_{2})$ is non-empty provided that 
$d(x_{1}$, $x_{2})\leq $ $r_{1}\cdot \overline{r_{2}}$. We say that 
$E$ is  \textit{hyperconvex} if the intersection
of every family of balls ($B_{E}(x_{i}$, $r_{i})$)$_{i\in I}$ is non-empty
whenever $d(x_{i}$, $x_{j})\leq $ $r_{i}\cdot
\overline{r_{j}}$ for all $i,j\in I$.
For an example, 
\emph{$(V,d_{V})$ is a hyperconvex  $V$-metric space
and every $V$-metric space embeds isometrically into a
power of $(V,d_{V})$} \cite{JaMiPo}.  This is due to the fact that for
every $V$-metric space $(E,d)$ and for all $x,y\in E$ the
following equality holds:
\begin{equation}\label{metricsup}
d\left( x,y\right)  = \underset{
z\in E}{\bigvee } d_{V} (d(z,x),d(z,y)).
\end{equation}
The space $E$ is a \textit{retract} of $E^{\prime }$, in symbols $
E\vartriangleleft E^{\prime }$, if there are two non-expansive maps $f:
E\longrightarrow E^{\prime }$ and $g:E^{\prime
}\longrightarrow E$ such that $g\circ f=id_{E}$ (where $id_{E}$ is
the identity map on $E$). In this case,
$f$ is a  \textit{coretraction} and $g$ a \textit{%
retraction}. If $E$ is a subspace of $E^{\prime }$, then clearly $E$ is a retract of
$E^{\prime }$ if there is a non-expansive map from $E^{\prime }$ to $E$ such $%
g(x)=x $ for all $x$ $\in E.$ We can easily see that every coretraction is
an isometry.  A metric space is an \textit{absolute retract} if it is a
retract of every isometric extension. The space $E$ is said to be \textit{%
injective} if for all $V$-metric space $E^{\prime }$ and $E'', $ each
non-expansive map $f:E^{\prime }\longrightarrow E$ and every isometry $g:E^{\prime
}\longrightarrow E''$ there is a non-expansive map $h:E''\longrightarrow E$ such that $%
h\circ g=f$.  We recall that \emph{for a metric space over a Heyting algebra $V$,
the notions of absolute retract, injective, hyperconvex and retract of a
power of $(V,d_{V})$ coincide} \cite{JaMiPo}.

A non-expansive map $f: E\longrightarrow E'$ is \textit{essential} it for
every non-expansive map  $g: E'\longrightarrow E''$,  the map $g\circ f$ is
an isometry if and only if $g$ is isometry (note that, in
particular, $f$ is an
isometry). An essential non-expansive map $f$ from $E$ into an injective $V$-metric space $E'$ is called an {\it injective envelope} of $E$. We will
rather say that $E'$ is  \textit{an injective envelope}  of  $E$. The construction of  
injective envelopes is based upon the notion of\textit{
minimal metric form}. A \emph {weak metric form} is every map $
f:E\longrightarrow V$ satisfying $d_{V}(d(x,y),f(y))\leq f(y)$ for all $
x,y\in E.$ This is a \emph{metric form} if in addition $f(x) \leq d(x,y)\cdot f(y)$ for all $
x,y\in E.$  A (weak) metric form is  \textit{minimal} if there is no other  (weak) metric form $g$
satisfying $g\leq f$ (that is $g(x)\leq f(x)$ for all $x\in E$). Since every weak metric form majorizes a metric form,  the two notions of minimality coincide.  As shown in
\cite{JaMiPo} \emph{every $V$-metric space
has an injective envelope, namely the space of
minimal metric forms (cf. also Theorem 2.2 of \cite{KP2})}. From this result follows that 
 an injective envelope of a metric space $E$ is  a minimal
injective $V$-metric space containing (isometrically) $E$. 
We will use particularly the following fact:
\begin{lemma}\label{fact:fixe}
If a non-expansive map from an injective envelope of $E$ into itself  fixes $E$ pointwise it is the identity map.
\end{lemma}

We also note that  two injective envelopes of $E$ are   isomorphic via an isomorphism which is the identity over $E$. This allows to talk about "the" injective envelope of $E$; we will denote it by $\mathcal N(E)$. A particular injective envelope of $E$ will be called a \emph{representation} of $\mathcal N(E)$.


We include the few facts we need about injective envelopes
of two-element metric spaces ( see \cite{KP2} for proofs).

Let $V$ be a Heyting algebra and $v\in V$. Let $E:=\{x,y\}$ be a two-element $V$-metric space such that $d(x,y)=v$. We denote by $\tilde {\mathcal{N}}_v$  the  injective envelope of $E$. We give two representations of it.  Let   $\mathcal{C}_{v}$ be the set of all pairs
$(u_{1},u_{2})\in V^{2}$ such that $v\leq u_{1}\cdot  \overline {u_{2}} $. Equip this
set with the ordering induced by the product ordering on $V^{2}$ and denote by $%
\mathcal{N}_{v}$ the set of its minimal elements. Each element of $\mathcal{N}_{v}$ defines a  minimal metric form. We equip $V^{2}$
with the supremum distance:
$$d_{V^{2}}\left( (u_{1},u_{2}), (u'_{1},u'_{2})\right)
:=d_{V}(u_{1}, u'_{1})\vee d_{V}(u_{2}, u'_{2}).$$

Let $v\in V$ and $\mathcal{S}_{v}:=  \left\{ \lceil v\cdot (\beta)^{-1}\rceil 
:\beta \in V\right\}$  be the subset of $V$;  equipped with the ordering induced by
the ordering over $V$ this is  a complete lattice. According to Lemma 2.5 of \cite{KP2},  $(x_1, x_2)\in \mathcal{N}_{v}$ iff  $x_1= \lceil v\cdot (x_2)^{-1} \rceil$ and $\overline {x_2}=\lceil (x_1)^{-1}\cdot v \rceil$. This yields a correspondence between   $ \mathcal {N}_{v}$ and  $\mathcal{S}_{v}$. 
\begin{lemma} (Lemma 2.3,  Proposition 2.7 of \cite {KP2})
The space $\mathcal{N}_{v}$ equipped with the supremum
distance and the set $\mathcal{S}_{v}$  equipped with
the distance induced by the distance over $V$ are injective envelopes of   the two-element metric spaces $\{({\bf 1}, v), (v, {\bf 1})\}$ and  $\left\{{\bf 1},v\right\}$ respectively. These spaces are isometric to the injective envelope of $E:= \{x,y\}$. 
\end{lemma}
\subsection{Composition of metric spaces}
Let $(E_{1},d_{1})$ and $(E_{2},d_{2})$ be two disjoint $V$-metric spaces;  let $x_{1}\in E_{1}$ and  $x_{2}\in E_{2}$. If we endow the set $\{x_{1}, x_{2}\}$ with
a $V$-distance $d^{\prime }$, then we can define a $V$-distance $d$ on $%
E:=E_{1}\cup E_{2}$ as follows:

$\bullet$ If $x,y\in E_{i}$  with $i\in \left\{
1, 2\right\}$  then $d(x,y) = d_{i}(x,y)$; 

$\bullet$ If $x\in E_{i}$, $y\in E_{j}$ with $i,j\in \left\{ 1, 2\right\} $ and $i\neq
j$,  then 
$d(x,y) = d_{i}(x,x_{i})\cdot d^{\prime }(x_{i},x_{j})\cdot  d_{j}(x_{j},y).$ In particular, we can identify $x_{1}$ and $x_{2}$ which amounts to set $%
d^{\prime }(x_{1},x_{2})={\bf 1}$ in the above formula.

If $E_{1}$ and $E_{2}$ are not disjoint, we replace it by two disjoint
copies $E_{1}^{\prime }$, $E_{2}^{\prime }($eg $E_{i}^{\prime
}: =E_{i}\times \left\{ i\right\} )$. Identifying the corresponding elements $x_{1}^{\prime },x_{2}^{\prime
} $, we obtain a $V$-metric space that we denote $%
(E_{1},d_{1})\cdot(E_{2},d_{2})$.  Alternatively, we may suppose that $E_{1}$ and $E_{2}$ have only one element in common, say $z_{1,2}$, and we define the distance $d$ on $E_{1}\cup E_{2}$ by setting $d(x,y):= d_i(x,z_{1, 2})\cdot d_j(z_{1,2},y)$ if $x\in E_i$, $y\in E_j$, $i\not =j$, and $d(x,y):=d_i(x,y)$ if $x,y\in E_i$. 

\begin{remark}\label{remark}If $E= E_1 \cdot E_2$ is injective then  $E_1$ and $E_2$ are retract of $E$ and hence  they are injective. The converse holds if $V= F(A^{\ast})$ (see Theorem \ref{prop: collage}). 
\end{remark}
We say that a  metric space  $E$ is \emph{irreducible} if it has more than one element and $E= E_1\cdot E_2$ implies $E= E_1$ or $E= E_2$. 

%
\subsection{Transition systems as metric spaces} We refer to \cite{sakarovitch}.  Let $A$ be a set. 
A \textit{transition system} on the \emph{alphabet} $A$ is a pair $M:=(Q$,
$T)$ where $T\subseteq Q\times A\times Q.$ The elements of $Q$ are
called \textit{states} and those of $T$ \textit{transitions}. Let
$M:=\left(Q,T\right) $ and $M^{\prime }:=\left( Q^{\prime
},T^{\prime }\right) $ be two transition systems on the alphabet
$A$. A map $f:Q\longrightarrow Q^{\prime }$ is a 
\textit{morphism} of transition systems if for every transition
$(p,\alpha,q)\in T$, we have $(f(p),\alpha
, f\left( q\right)) \in T^{\prime }$. When $f$ is bijective and $%
f^{-1} $ is a morphism from $M^{\prime }$ to
$M$, we say that $f$ is an \textit{isomorphism}.
The collection of transition systems over $A$, equipped with these morphisms,
form a category and  this category has products. The \emph{graph of a transition system}  $M:= (Q, T)$ is the directed graph  with vertex set $Q$ and arcs $(x,y)$ such that $(x, a, y)\in T$ for some $a\in A$ .

%
An \textit{
automaton} $\mathcal A$ on the alphabet $A$ is given by a transition system
$M:=\left( Q,T\right)$ and two subsets $I,$ $F$ of $Q$ called the sets
of \textit{initial} and \textit{final states}. We denote the
automaton as a triple $\left( M,I,F\right)$.  A \textit{path} in
the automaton $\mathcal{A}:=\left( M,I,F\right)$ is a sequence
$c:=\left( e_{i}\right) _{i<n}$ of consecutive
transitions, that is  of transitions $e_{i}:=(q_{i},a _{i},q_{i+1})$.
The  word $\alpha:=a_{0}\cdots a_{n-1}$ is the \textit{label} of the path, the 
state $q_{0}$ is its \textit{origin} and the state $q_{n}$ its \textit{end}.
One agrees to define for each state $q$ in $Q$ a unique null path
of length $0$ with origin and end $q$. Its label is the empty word
$\Box$. A path is \textit{successful} if its origin is in $I$ and
its end is in $F$. Finally,  a word $\alpha$ on the alphabet $A$ is
\textit{accepted} by the automaton
$\mathcal{A}$ if it is the label of some successful path. The \textit{%
language accepted} by the automaton $\mathcal{A}$,  denoted by
$L_{\mathcal{A}}$,  is the set of all words accepted by
$\mathcal{A}$. Let $\mathcal{A}:=\left( M,I,F\right) $ and
$\mathcal{A}^{\prime }:= \left( M^{\prime },I^{\prime },F^{\prime
}\right) $ be two automata. A {\it morphism} from $\mathcal{A}$ to $\mathcal{A}'$ is a map $f:Q\longrightarrow Q^{\prime }
$ satisfying the two conditions:
\begin{enumerate}
\item  $f$ is  morphism from $M$ to $M^{\prime }$; 
\item  $f$ $(I)$ $\subseteq I^{\prime }$ and $f(F)\subseteq F^{\prime}$. 
\end{enumerate}
If, moreover, $f$ is bijective, $f(I)=I^{\prime },f(F)=F^{\prime }$ and $%
f^{-1}$ is also a  morphism from $\mathcal{A}^{\prime }$ to $%
\mathcal{A}$, we say that $f$ is an  \textit{isomorphism} and that
the two automata $\mathcal{A}$ and $\mathcal{A}^{\prime }$ are \textit{%
isomorphic}.

According to Lemma \ref{fact:heyting},  $F(A^{\ast})$ is a Heyting algebra (we will sometimes denote $F\cdot F'$ the concatenation  of $F$ and $F'$).  Hence,  we may  consider metric spaces over $V:= F(A^{\ast})$. To a metric space $\left( E,d\right) $ over
$V:= F(A^{\ast })$, we may associate the transition system
$M:=\left( E,T\right) $ having $E$ as   set of states and
$T:=\left\{ \left( x,a,y\right) :a\in d\left( x,y\right) \cap
A\right\} $ as   set of transitions. Notice that such a
transition system has the following properties: for all $x,y\in E$
and every $a,b\in A$ with $b\geq a$:\\
1) $\left( x,a,x\right) \in T$; \\
2) $\left( x,a,y\right) \in T$ implies $\left(
y,\overline{a},x\right) \in T$; \\
3) $\left( x,a,y\right) \in T$ implies  $\left( x,b,y\right) \in
T.$\\
We say that a transition system satisfying these properties is \textit{%
reflexive} and \textit{involutive} (cf. \cite{Sa}, \cite{KP2}). Clearly \ if $%
M:=\left( Q,T\right) $ is such a transition system, the map
$d_{M}:Q\times Q\longrightarrow F(A^{\ast })$ where $d_{M}\left(
x,y\right) $ is the language accepted by the automaton $\left(
M,\left\{ x\right\},\left\{ y\right\} \right) $ is a distance. The graph of $M$ is reflexive and symmetric. We
have the following:

\begin{lemma}\label{lem:metric-transition}
Let $\left( E,d\right)$ be a metric space over
$F\left( A^{\ast }\right)$.  The following properties are
equivalent: 
\begin{enumerate}[(i)]
\item The map $d$ is of the form $d_{M}$ for some
reflexive and involutive transition system $M:=(E,T)$; 
\item  For all $\alpha,\beta \in A^{\ast }$ and  $x$,
$y$ $\in E$, if $\alpha\beta \in d\left( x,y\right)$, then there
is some $z\in E$ such that $\alpha \in d\left( x,z\right)$ and
$\beta \in d\left( z,y\right)$.\end{enumerate}
\end{lemma}

The category of reflexive and involutive transition
systems with the morphisms defined above identify  to a subcategory of the
category having as objects the metric spaces and morphisms the non-expansive maps. Indeed:

\begin{lemma}\label{fact:morphism}  Let $M_{i}:=\left( Q_{i},T_{i}\right) \left( i=1,2\right) $
be two reflexive and involutive transition systems. A map $%
f:Q_{1}\longrightarrow Q_{2}$ is a morphism from $M_{1}$ to $M_{2}$ if only
if $f$ is a non-expansive map  from $(Q_{1},d_{M_{1}})$ to $%
(Q_{2},d_{M_{2}}).$
\end{lemma}


Injective objects satisfy the convexity property stated in $(ii)$ of Lemma \ref{lem:metric-transition}. In particular, if $F$ is a final segment of $A^{\ast }$, the distance on the injective envelope  $\mathcal N_F$ comes from a transition system. Moreover, if  $A$ is well-quasi-ordered then from Higman theorem \cite{higman}, the final
segment $F$ has a finite basis, that is, there are finitely many
words $\alpha _{0},...,\alpha _{n-1}$ such that $F=\uparrow \{
\alpha _{i}:  i<n\}$. In particular, we get:

\begin{theorem}\label{envelopeF}
For every $F\in F(A^{\ast})$ there is a transition system $M:=(Q,T)$, an initial state $x$ and a
final state $y$ such that the language accepted by the automaton
$\mathcal{A}=(M,\{x\},\{y\})$ is $F.$ Moreover, if $A$ is
well-quasi-ordered then we may choose $Q$ to be finite.
\end{theorem}

Let $M_i:=(Q_i,T_i)$, resp. $G_i:=(Q_i, \mathcal E_i)$, $(i=1,2)$,  be two transition systems,  resp. graphs.
Let us suppose that they have exactly one state, resp. one vertex,  in common, say $x$.
We denote by $M_1\cdot M_2$, resp. $G_1\cdot G_2$,   the transition system $M:=(Q,T)$, resp. graph $G:= (Q, \mathcal E)$,  such that
$Q:=Q_1 \cup Q_2$ and $T:=T_1 \cup T_2$, resp. $\mathcal E:= \mathcal E_1\cup \mathcal E_2$. 


The following lemma is immediate:
 
 \begin{lemma}\label{fact:product}
Let  $M_i:=(Q_i,T_i)$, $(i=1,2)$ be  two transition systems having $x$ as
the only state in common. If $E_i$ and $G_i$  are the metric space  and graph  corresponding to $M_i$ $(i=1,2)$ then  $E_1\cdot E_2$ and $G_1\cdot G_2$ are the metric space and graph corresponding to  $M_1\cdot M_2$. 
\end{lemma}
We recall the following results of  \cite{KP2}:
\begin{theorem}(Proposition 4.7 p. 175)\label{prop: collage}
Let $M_i:=(Q_i,T_i)$, $(i=1,2)$,  be two transition systems having $x$ as
the only state in common and $M:= M_1\cdot M_2$. 
If the space $E_i:=(Q_i, d_{M_i})$ ($i=1,2$)  is injective, then $E_1\cdot E_2$ is injective. \end{theorem}
\begin{corollary} (Corollary 4.9. p. 177)\label{cor:productinjective} Let $F_1$ and $F_2$ be two final segments of $A^{\ast}$. If $F_1F_2$ is non empty then $\mathcal {S}_{F_1F_2}$ is isomorphic to $\mathcal {S}_{F_1}\cdot\mathcal {S}_{F_2}$. 
\end{corollary}

The reader will realize that these two results expresse in terms of metric spaces the fact that $F(A^{\ast})$ satisfies the decomposition property. 

%
%
%
%
%
%
%
\subsection{\bf Proof of Theorem \ref{thm:blockdecomposition}} \label{subsection:graph}We refer to \cite{bondy-murty} for notions of graph theory, particularly to Chapter 5, for the notions of cut vertex and block decomposition. The graphs we consider are simple, with a  loop at every vertex, and can be infinite. A \emph{cut vertex} $x$ of a graph $\mathcal G$ is any vertex whose deletion increases the number of connected components of $\mathcal  G$ (hence if $\mathcal G$ has no edge, no vertex is a cut vertex); a \emph{block} is a maximal connected induced subgraph  with no cut vertex (since our graphs are reflexive, we prefer this definition to the usual one); any two blocks have at most one vertex in common; if $\mathcal G$ is connected with more than a vertex, the blocks of $\mathcal G$ induce  a decomposition of the edge set of $\mathcal G$ and are the vertices of a tree (cf. Proposition 5.3, p. 120 of \cite {bondy-murty}).

Let $F$ be  a final segment of $A^{\ast}$, let $E= (\{x,y\}, d)$ be a $2$-element metric space  such that $d(x,y)=F$ and $\tilde{\mathcal N}_F$ be its injective envelope.  With no loss of generality, we may suppose that $x= A^{\ast}$,
 $y=F$ and $\tilde{\mathcal N}_F= \mathcal S_F$. Let $\mathcal {M}_F$ be the transition system associated with  $\mathcal S_F$, let $Q$ be  its domain  and  $\mathcal G_F$ be the graph of this transition system. 
 
 We suppose that $F\not = A^{\ast}$, hence $x\not =y$. We  prove first that $F$ is irreducible if and only if  $\mathcal S_F$ is irreducible. If $F$ is not irreducible then there are  two final segments $F_1$ and $F_2$ distinct from 
$F$ such that $F=F_1F_2$. Necessarily, $F, F_1$ and $F_2$ are non-empty. According to Corollary \ref{cor:productinjective},  $\mathcal {S}_{F}$ is isomorphic to $\mathcal {S}_{F_1}\cdot\mathcal {S}_{F_2}$, hence $\mathcal S_F$ is not irreducible. Conversely, suppose that $\mathcal S_F$ is not irreducible. Let $E_1$ and $E_2$ such that $\mathcal S_F= E_1\cdot E_2$ and $z$ in their intersection. First, $x$ and $y$ do not belong  to the same  $E_i$, otherwise  we may retract $\mathcal S_F$ onto  $E_i$ by a non-expansive map sending $E_j$ ($j\not =i$) onto $z$, contradicting Lemma  \ref{fact:fixe}. Suppose $x\in E_1$ and $y\in E_2$. From Lemma \ref{fact:product},  we have $F= d_{\mathcal S_F}(x,y)= d_{E_1}(x,z)\cdot  d_{E_2}(y,z)$ hence $F$ is not irreducible. 

We prove now that $\mathcal S_F$ is irreducible if and only if 
$\mathcal G_F$ has no cut vertex.  
If $\mathcal S_F$ is not irreducible then $\mathcal S_F=E_1\cdot E_2$ for two proper subspaces of 
$\mathcal S_F$. Let $\mathcal M_i$ be the restriction of $\mathcal M_F$ to $E_i$ $(i=1,2)$. We claim that $\mathcal M= \mathcal M_1\cdot \mathcal M_2$.  Since $\mathcal S_F$ is injective, the distance $d_{\mathcal S_F}$ is equal to  the distance $d_{\mathcal M_F}$ (Lemma \ref {lem:metric-transition}).  By Remark   \ref{remark}, $E_1$ and $E_2$ are injective, hence the distance induced on $E_i$ coincides with the the distance  $d_{\mathcal M_i}$. Let $z$ with $\{z\}=E_1\cap E_2$,  $x'\in E_1\setminus \{z\}$ and $y'\in E_2\setminus \{z\}$. Since $\mathcal S_F=E_1\cdot E_2$, $d_{\mathcal S_F}(x',y')= d_{\mathcal S_F}(x',z)\cdot d_{\mathcal S_F}(z,y')$. Hence, there is no transition, thus no edge,  linking $x'$ and $y'$. This proves our claim. In particular, $z$ is a cut vertex of  $\mathcal G_F$. 

Suppose that $\mathcal {G}_F$ has a cut vertex $z$. Then $F= d(x,y)\not =\emptyset$ otherwise $\mathcal S_F=E$ and $\mathcal G_F$ has no cut vertex. 
%
%
%
%
%
We claim that since $F\not = \emptyset$, $x$ and $y$ are in the same connected component. Furthermore $\mathcal {G}_F$ is connected, neither $x$ nor $y$ is a cut vertex and every cut vertex $z$ separates $\mathcal G_F$ into two connected components, one containing $x$, the other $y$. The proof of this claim uses repeatedly Lemma \ref{fact:fixe}. If one of these assertions is false, we can define  a proper non-expansive retraction of  $\mathcal S_F$ which fixes $x$ and $y$. According to Lemma \ref{fact:fixe},  it fixes $\mathcal S_F$, contradicting the fact that it is proper. To illustrate, suppose that $z$ in a cut vertex of $\mathcal G_F$ distinct from $x$ and $y$. Let $D$ be the union of connected components containing $x$ and $y$. If there are other connected components we can retract  these components  on $z$.  Since $\mathcal {M}_F$ is reflexive   this retraction is a retraction of $\mathcal {M}_F$ onto its restriction to $D\cup \{z\}$. It induces a non-expansive map from $\mathcal {S}_F$ onto itself which fixes $x$ and $y$. According to Lemma \ref{fact:fixe},  it fixes $Q$, hence $Q= D\cup \{z\}$  contradicting the existence of other connected. components.  Since $z$ is a cut vertex  $D$ consists of two connected components $D_x$ and $D_y$. The sets $D_{x}\cup \{z\}$ and  $D_{y}\cup \{z\}$ form a covering of $Q$ into two connected subsets with no crossing edge, hence $\mathcal M_F= \mathcal M_{\restriction D_x}\cdot \mathcal M_{\restriction D_y}$. According to Fact \ref{fact:product},  $\mathcal S_F=\mathcal S_{ \restriction D_x}\cdot \mathcal S_{ \restriction D_y}$ hence $\mathcal S_F$ is not irreducible.

Suppose that $F$ is not irreducible. In this case $F$ is non-empty.  Hence ${\mathcal G}_F$ is connected.  Since $\mathcal G_F$ has a cut vertex, it has at least two blocks. The collection of blocks forms a tree. Let $C$ be the shortest path joining the block containing $x$ to the block containing $y$ and let $\tilde C$ be the graph induced on the union of blocks belonging to $C$. Since $\mathcal G_F$ is a tree with a loop at every vertex, we may retract  $\mathcal G_F$ on $\tilde C$ by a map fixing pointwise the vertices in $\tilde C$ (send each vertex $z\in \mathcal {G}_F$ on the closest vertex belonging to $\tilde C$). Since $\mathcal M_F$ is reflexive, this retraction is  a retraction from $\mathcal M_F$ onto the transition system induced  on $\tilde C$ and thus a retraction of the injective  enveloppe $\mathcal S_F$ onto the space induced on $\tilde C$. Since this retraction fixes $x$ and $y$, it fixes $Q$ (Lemma \ref{fact:fixe}), hence $Q= \tilde C$. We can enumerate the vertices of $C$ in a  sequence $C_0, \dots, C_{n-1}$ with $x\in C_0$ and $y\in C_{n-1}$, with $n\geq 2$. Let $F_i$ be the language accepted by the automaton $({\mathcal M}_{F}\restriction C_{i},\left\{ x_i\right\}, \left\{y_i\right\} )$, where $x_i:=x$ if $i=0$, $y_{i}= y$ if $i=n-1$ and  $\{x_i\}=C_{i-1} \cap C_{i}$, $\{y_i\}= C_{i} \cap C_{i+1}$, otherwise. Clearly, $F$ is the product $F_0 \dots F_{n-1}$. Also, $\mathcal S_{F}\restriction C_{i}$ is the metric space associated with the injective envelope of $(\{x_i, y_i\}, d_i)$ where $d_i(x_i, y_i)= F_i$. With this the proof is complete.  
\hfill $\Box$

\end{document}